\documentclass[11pt,reqno,oneside]{amsart}

\usepackage{graphicx}
\usepackage{lscape}
\usepackage{amsfonts}
\usepackage{pdfsync}
\usepackage[toc,page]{appendix}
\usepackage{hyperref}

\newtheorem*{theorem-non}{Theorem}
\newtheorem{theorem}{Theorem}[section]
\newtheorem{lemma}[theorem]{Lemma}
\newtheorem{corollary}[theorem]{Corollary}

\theoremstyle{definition}
\newtheorem{defi}[theorem]{Definition}

\newtheorem{theorem*}[theorem]{Theorem}

\def\O{\Omega}

\def\o{\omega}

\def\d{{\rm d}}

\def\diam{{\rm diam}}
\def\dx{\,{\rm d}x}
\def\dy{\,{\rm d}y}
\def\dz{\,{\rm d}z}

\def\tf{\tilde{f}}
\def\tg{\tilde{g}}
\def\tgc{\tg^{(0)}}
\def\th{\tilde{h}}
\def\thc{\th^{(0)}}
\def\thu{\th^{(1)}}
\def\tt{\tilde{\theta}}
\def\tvfi{\tilde{\varphi}}
\def\tpsi{\tilde{\psi}}

\def\tF{\tilde{F}}

\def\tnu{\tilde{\nu}}
\def\W{\mathcal{W}}
\def\V{\mathcal{V}}
\def\S{\mathcal{S}}
\def\Rnn{\R^{n\times n}}

\newcommand{\rr}{{\bf r}}
\newcommand{\uu}{{\bf u}}
\newcommand{\vv}{{\bf v}}
\newcommand{\ww}{{\bf w}}

\newcommand{\R}{{\mathbb R}}

\newcommand{\bs}{\bigskip}

\begin{document}
\date{\today}
\section*{}
\title[]{Weighted generalized Korn inequality on John domains}

\author[F. L\'opez Garc\'\i a]{Fernando L\'opez Garc\'\i a}
\address{University of California Riverside, Department of Mathematics, 
900 University Ave. Riverside (92521), CA, USA} \email{fernando.lopezgarcia@ucr.edu}



\begin{abstract} The goal of this work is to show that the generalized Korn inequality that replaces the symmetric part of the differential matrix in the classical Korn inequality by its trace-free part is valid over John domains and weighted Sobolev spaces. The weights considered are nonnegative powers of the distance to the boundary. 

\end{abstract}

\maketitle

\section{Introduction}
\label{intro}
\setcounter{equation}{0}

Let $\O\subset \R^n$ be a bounded domain and $1<p<\infty$.  A version of Korn inequality, equivalent to the classical one, states that 
\begin{align}\label{ClassKorn2}
\inf_{\varepsilon(\ww)=0}\left(\int_\Omega |D\vv-D\ww|^p\, \d x\right)^{1/p} \leq C \left(\int_\Omega |\varepsilon(\vv)|^p\, \d x\right)^{1/p},
\end{align}
for any arbitrary vector field $\vv$ in the Sobolev space $W^{1,p}(\Omega,\R^n)$ where $n\geq 2$. The constant $C$ depends only on $\Omega$ and $p$, and $\varepsilon(\vv)$ denotes the symmetric part of the differential matrix $D\vv$, namely,  
\[\varepsilon(\vv):=\frac{D\vv+{D\vv}^T}{2}.\]

This inequality is a fundamental result in the analysis of linear elasticity equations and has been widely studied by several authors since the seminal works by A. Korn published in 1906 and 1909. In this context, the vector field $\vv$ represents the displacement of an elastic body and $\varepsilon(\vv)$ the linear part of the strain tensor.

It is well-known that the validity of the classical Korn inequality depends on the geometry of the domain. For example, inequality (\ref{ClassKorn2}) fails on certain domains with cusps (see \cite{ADL,W}). The largest known class of domains where this inequality is satisfied is the class of John domains (see \cite{ADM,DRS,L2}). This class was introduced by Fritz John in \cite{Joh} and named after him by Martio and Sarvas in \cite{MS} and contains convex domains, Lipschitz domains and even domains with fractal boundary such as the Kock snowflake. 

In this work we deal with a generalized version of (\ref{ClassKorn2}) known simply as generalized Korn inequality or conformal Korn inequality, where the linearized strain vector $\varepsilon(\vv)$ is replaced by its trace-free part:
\[l(\vv):=\varepsilon(\vv)-\frac{\text{div}\,\vv}{n} I,\]
where $I$ in $\R^{n\times n}$ is the identity matrix. More specifically, the main goal of these notes is proving that the inequality 
\begin{align}\label{ConKorn}
\inf_{l(\ww)=0}\left(\int_\Omega |D\vv-D\ww|^p\, \d x\right)^{1/p} \leq C \left(\int_\Omega | l(\vv)|^p\, \d x\right)^{1/p},
\end{align}
is valid for any vector field $\vv$ in $W^{1,p}(\Omega,\R^n)$, where $\O\subset\R^n$ is an arbitrary bounded John domain and $n\geq 3$.

Different types of Korn inequality involving the trace-free part of the symmetric gradient have been recently studied for their interest as a mathematical result and for their applications, for instance, to general relativity and Cosserat elasticity. See the following articles and references therein \cite{FZ,FS,Sch}. However, we are specially interested in two Korn-type inequalities published in \cite{R1} and \cite{Da} where no assumptions on the values of the vector fields over the boundary of the domain are considered. In the first article, Yu. G. Reshetnyak showed the following generalized Korn inequality over start-shaped domains with respect to a ball in $\R^n$, where $n\geq 3$ and $1<p<\infty$: 
\begin{equation}\label{ReshCons}
\|\vv-\Pi(\vv)\|_{W^{1,p}(\O)}\leq C\|l(\vv)\|_{L^p(\O)},
\end{equation}
valid for all $\vv\in W^{1,p}(\O,\R^n)$. The operator $\Pi:W^{1,p}(\O,\R^n)\to \Sigma$ is a projection (i.e. a continuous linear operator such that $\Pi(\vv)=\vv$ for all $\vv\in\Sigma$), where $\Sigma$ is the kernel of $l$ and it is endowed with the topology of $L^p(\O,\R^n)$. This result was proved by using a certain integral representation of the vector field $\vv$ in terms of $l(\vv)$ and then the theory of singular integral operators. Let us recall that the class of star-shaped domains with respect to a ball contains convex domains and it is strictly included in the class of John domains. The study of this inequality was motivated by its connexion with the stability of Liouville's theorem.

The second Korn-type inequality of our interest was published in \cite{Da} and says  
\begin{equation}\label{DainCons}
\|\vv\|_{W^{1,2}(\O)}\leq C\{\|(\vv)\|_{L^2(\O)}+\|l(\vv)\|_{L^2(\O)}\},
\end{equation}
where $\vv$ is an arbitrary vector field in $W^{1,2}(\O,\R^n)$ and $C$ depends only on $\O$. In this case $\O$ is an arbitrary bounded Lipschitz domain in $\R^n$, with $n\geq 3$. This theorem was proved in \cite{Da} by using the classical result known as Lions Lemma that claims that any distribution in the space $H^{-1}(\O)$ with gradient in $H^{-1}(\O)^n$ belongs to $L^2(\O)$. A generalized version of this result for distributions in $H^{-2}(\O)$ is also required. It is shown in \cite{Da} that inequality (\ref{DainCons}) fails on planar domains independently of the geometry of the domain.

The main result in this work states that generalized Korn inequality (\ref{ConKorn}) holds also in weighted Sobolev spaces on John domains where the weights are nonnegative powers of the distance to the boundary. 

\begin{theorem}\label{gkJohn} Let $\Omega\subset\R^n$ be a bounded John domain with $n\geq 3$, $1<p<\infty$ and $\beta\geq 0$. Then, there exists a constant $C$ such that
\begin{align}\label{wgkJohn}
\inf_{l(\ww)=0}\left(\int_\Omega |D\vv-D\ww|^p\rho^{p\beta}\, \d x\right)^{1/p} \leq C \left(\int_\Omega |l(\vv)|^p\rho^{p\beta}\, \d x\right)^{1/p}
\end{align}
for all vector field $\vv\in W^{1,p}(\Omega,\R^n,\rho^\beta)$. 
The function $\rho(x)$ is the distance from $x$ to the boundary of $\Omega$.
\end{theorem}

Let us recall the definition of John domains. A bounded domain $\Omega\subset\R^n$, with $n\geq 2$, is called a {\it John domain} with parameter $C>1$ if there exists a point $x_0\in\O$ such that every $y\in\O$ has a rectifiable curve parameterized by arc length $\gamma:[0,l]\to\Omega$ such that $\gamma(0)=y$, $\gamma(l)=x_0$ and 
\begin{align*}
\text{dist}(\gamma(t),\partial\O)\geq \frac{1}{C}t
\end{align*}
for all $t\in[0,l]$, where $l$ is the length of $\gamma$.

Finally, as a Corollary of the main theorem we prove that (\ref{ReshCons}) and (\ref{DainCons}) are also valid on bounded John domains.

\section{Notation}
\label{Prel}
\setcounter{equation}{0}

Throughout the paper, $\Omega\subset\R^n$ is a bounded domain with $n\geq 3$ and $1<p,q<\infty$ with $\frac{1}{p}+\frac{1}{q}=1$. Given $\eta:\O\to\R$ a positive measurable function we denote by $L^p(\Omega,\eta)$ the space of Lebesgue measurable functions $u:\Omega\to\R$
equipped with the norm:
\[\|u\|_{L^p(\Omega,\eta)}:=\left(\int_\Omega|u(x)|^p\eta^p(x)\,\d x\right)^{1/p}.\]
Analogously, we define the weighted Sobolev spaces $W^{1,p}(\Omega,\eta)$ as the
space of weakly differentiable functions $u:\Omega\to\R$ with
the norm:
\[\|u\|_{W^{1,p}(\Omega,\eta)}:=\left(\int_\Omega|u(x)|^p\eta^p(x)\,\d x
+\sum_{i=1}^n\int_\Omega\left|\frac{\partial u(x)}{\partial
x_i}\right|^p\eta^p(x)\,\d x\right)^{1/p}.\]
We extend this definition to function from $\O$ to $\Rnn$ and from $\O$ to $\R^n$ denoted by $L^p(\Omega,\Rnn,\eta)$ and $W^{1,p}(\Omega,\R^n,\eta)$, respectively.

Given $\tg:\R^n\to \R^{n\times n}$ and $1\leq r \leq\infty$ we denote by $\|\tg\|_r:\R^n\to\R$ the following function
\[\|\tg\|_r(x):=\left(\sum_{1\leq i,j\leq n} |\tg_{ij}(x)|^r\right)^{1/r}\]
if $r\neq \infty$, and
\[\|\tg\|_\infty(x):=\max_{1\leq i,j\leq n} |\tg_{ij}(x)|.\] Notice that for any $1\leq r_1,r_2\leq \infty$ there is a positive constant $C=C(r_1,r_2)$ such that 
\begin{equation*}
\frac{1}{C}\|\tg\|_{r_1}(x)\leq \|\tg\|_{r_2}(x)\leq C \|\tg\|_{r_1}(x)
\end{equation*}
for all functions $\tg$ and $x\in\O$. 

Moreover, $|\tg|:\R^n\to\R^{n\times n}$ is given by $|\tg|_{ij}(x)=|\tg_{ij}(x)|$. We say that the function $\tg$ is integrable (similarly bounded) if each coordinate is integrable (bounded). 

For $\tg,\tf:\R^n\to\R^{n\times n}$ we denote by $\tg(x):\tf(x)$ the product coordinate by coordinate 
\begin{align*}
\tg(x):\tf(x):=\sum_{1\leq i,j \leq n} \tg_{ij}(x)\tf_{ij}(x).
\end{align*}
We say that $x$ belongs to $supp(\tg)$ iff $\tg(x)$ has at least one coordinate different from zero. We denote with tilde those functions with codomain in $\Rnn$. 

Finally, a Whitney decomposition of $\O$ is a collection $\{Q_t\}_{t\in\Gamma}$ of closed dyadic cubes whose interiors are pairwise disjoint, which verifies
\begin{enumerate}
\item $\O=\bigcup_{t\in\Gamma}Q_t$,
\item $\text{diam}(Q_t) \leq \text{dist}(Q_t,\partial\Omega) \leq 4\text{diam}(Q_t)$,
\item $\frac{1}{4}\text{diam}(Q_s)\leq \text{diam}(Q_t)\leq 4\text{diam}(Q_s)$, if $Q_s\cap Q_t\neq \emptyset$.
\end{enumerate}
Two different cubes $Q_s$ and $Q_t$ with $Q_s\cap Q_t\neq \emptyset$ are called {\it neighbors}. Notice that two neighbors may have an intersection with dimension less than $n-1$. For instance, they could be intersecting each other in a one-point set. We say that $Q_s$ and $Q_t$ are \mbox{$(n-1)$}-neighbors if $Q_s\cap Q_t$ is the $n-1$ dimensional face of one of them.
This kind of covering exists for any proper open set in $\R^n$ (see \cite{S} for details). Moreover, each cube $Q_t$ has less than $12^n$ neighbors. And, if we fix $0<\epsilon<\frac{1}{4}$ and define $Q_t^*$ as the cube with the same center as $Q_t$ and side length $(1+\epsilon)$ times the side length of $Q_t$, then $Q_t^*$ touches $Q^*_s$ if and only if  $Q_t$ and $Q_s$ are neighbors. 

Given a Whitney decomposition $\{Q_t\}_{t\in\Gamma}$ of $\O$ we refer by an extended Whitney decomposition of $\O$ to the collection of open cubes $\{\Omega_t\}_{t\in\Gamma}$ defined by 
\begin{align*}
\Omega_t:=\frac98 Q_t^\circ.
\end{align*}
Observe that this collection of cubes satisfies that \[\chi_\O(x) \leq \sum_{t\in\Gamma}\chi_{\O_t}(x)\leq 12^n \chi_\O(x)\]
for all $x\in\R^n.$

\section{Proof of the main result}
\label{proof}
\setcounter{equation}{0}

This section is devoted to show Theorem \ref{gkJohn}. The proof follows from a local-to-global argument based on the validity of (\ref{wgkJohn}), with $\beta=0$, on cubes and a certain decomposition of functions stated in Lemma \ref{DecompSimple} which is proved in Section \ref{V-decomp}.

The following result is implied by the validity of (\ref{ReshCons}) proved by Reshetnyak \cite{R1} over star-shaped domains with respect to a ball.
\begin{corollary}\label{ConKornCubes}
Let $Q\subset\R^n$ be an arbitrary cube with sides parallel to the axis, with $n\geq 3$, and $1<p<\infty$. Then, there exists a positive constant $C$ that depends only on $n$ and $p$ such that 
\begin{align}\label{infwKorn}
\inf_{l(\ww)=0}\left(\int_Q |D\vv-D\ww|^p\, \d x\right)^{1/p} \leq C \left(\int_Q |l(\vv)|^p\, \d x\right)^{1/p}
\end{align}
for all $\vv\in W^{1,p}(Q,\R^n)$.
\end{corollary} 
\begin{proof} Cubes are star-shaped domains with respect to a ball, thus from (\ref{ReshCons}) we can conclude that 
\begin{align*}
\inf_{l(\ww)=0} \|D\vv-D\ww\|_{L^p(Q)} \leq \|\vv-\Pi(\vv)\|_{W^{1,p}(Q)}\leq C\|l(\vv)\|_{L^p(Q)},
\end{align*}
where $C$ depends on $p$, $n$ and $Q$. It only remains to prove that there is a uniform constant that makes (\ref{infwKorn}) valid for any arbitrary cube with sides parallel to the axis. Let $Q_0$ be the cube $(0,1)^n$ with constant $C_{Q_0}$ in the inequality (\ref{infwKorn}). Hence, any other cube with sides parallel to the axis can be obtained by a translation and dilation of $Q_0$. Now, the inequality only involves partial derivatives of first order, thus by making a change of variable we can extend the validity of this Korn type inequality from $Q_0$ to any other cube in the sttatement of  corollary with the same constant $C_{Q_0}$. 
\end{proof}

The kernel of the operator $l$, denoted by $\Sigma$ in these notes, plays a central role in this local-to-global argument. So, let us recall its characterization which is significantly different if $n=2$ or $n\geq 3$. In the planar case, $\Sigma$ is an infinite dimensional space where $\ww\in\Sigma$ iff $\ww(x,y)=(w_1(x,y),w_2(x,y))$ where $w_1$ and $w_2$ are the components of an analytical function $F(x+iy):=w_1(x,y)+iw_2(x,y)$. The fact that the kernel has infinite dimension and the well-known Rellich-Kondrachov Theorem for Sobolev spaces imply the failure of (\ref{DainCons}) for planar domains (see \cite{Da}). We have described the planar case for general knowledge, however, in this article, we deal with $n\geq 3$. In that case , when $n\geq 3$, the kernel of $l$ has a finite dimension equal to $\frac{(n+1)(n+2)}{2}$ and a vector field $\ww\in\Sigma$ iff  
\begin{align}\label{Kernel}
\ww(x)=a+A(x-y)+\lambda (x-y)+\left\{\langle b,x-y\rangle (x-y)-\frac{1}{2}|x-y|^2b\right\},
\end{align}
where $A\in\R^{n\times n}$ is skew-symmetric, $a,b\in\R^n$ and $\lambda\in\R$. The vector $y\in\R^n$ is arbitrary but must be fixed to have uniqueness for this representation.

Now,  we define the space $\V$ which elements are the differential matrix of the vector fields in $\Sigma$. Namely,
\begin{equation*}
\begin{split}
\mathcal{V}:=\{\tvfi:\R^n\to\R^{n\times n}&:\tvfi(x)=A+\lambda I+\sum_{i=1}^n b_i H_i(x-y)\\ 
& \text{ with } A\text{ skew-symmetric, }\lambda\in\R\text{ and } b_i\in\R\text{ for all }i\}.
\end{split}
\end{equation*}
The matrix $I$ is the identity and, for $1\leq i\leq n$, $H_i(x)$ is the functions with values in $\R^{n\times n}$ defined by  
\begin{equation}\label{MatrHi}
(H_i)_{jk}(z)=\left\{
  \begin{array}{r l}
  z_i & \quad \text{if }j=k\\
  z_j & \quad \text{if }j\neq k,\ k=i\\
-z_k & \quad \text{if }j\neq k,\ j=i\\     
     0 & \quad \text{otherwise}\\
   \end{array} \right.
\end{equation}
for $1\leq j,k\leq n$. In the particular case when $n=3$, we have:
\begin{align*}
H_1(z)=\left( \begin{array}{rrr}
z_1 & -z_2 & -z_3 \\
z_2 & z_1 & 0 \\
z_3 & 0 & z_1 \end{array} \right)
\hspace{.2cm}
H_2(z)=\left( \begin{array}{rrr}
z_2 & z_1 & 0 \\
-z_1 & z_2 & -z_3 \\
0 & z_3 & z_2  \end{array} \right) 
\hspace{.2cm}
H_3(z)=\left( \begin{array}{rrr}
z_3 & 0 & z_1 \\
0 & z_3 & z_2 \\
-z_1 & -z_2 & z_3 \end{array} \right).
\end{align*}
Observe that the definition of $\V$ does not depend on the vector $y\in\R^n$. By taking a different $y$ we only obtain a different representation of the functions $\tvfi(x)$ in $\V$. Thus, let us denote by $m:=\frac{n(n-1)}{2}+1+n$ the dimension of $\V$. Finally, to prove that $\tvfi$ belongs to $\V$ iff $\tvfi=D\ww$ for some $\ww\in\Sigma$ it is sufficient to show that the quadratic part appearing between brackets in (\ref{Kernel}), denoted by $\rr$ for simplicity, is $\sum\limits_{i=1}^n b_i H_i(x-y)$. Indeed, 
\begin{equation*}\label{Hi}
\frac{\partial \rr_j}{\partial x_k}=\left\{
  \begin{array}{l l}
     \sum\limits_{i=1}^n b_i(x_i-y_i) & \quad \text{if }j=k\\
       & \\
     b_k(x_j-y_j)-b_j(x_k-y_k) & \quad \text{if }j\neq k.\\
   \end{array} \right.
\end{equation*}
Thus, $\left(\frac{\partial \rr_j}{\partial x_k}\right)_{jk}=\sum\limits_{i=1}^n b_i H_i(x-y)$ concluding that \[``\V=D(\Sigma)".\]

Now, we define the subspace $\W \subset L^q(\O,\Rnn,\rho^{-\beta})$ by: 
\begin{equation*}
\W:=\{\tg\in L^q(\Omega,\Rnn,\rho^{-\beta}): \int \tg:\tvfi=0\text{ for all }\tvfi\in \V\},
\end{equation*}
where the product ``$:$'' is the standard inner product for vectors understanding matrices in $\R^{n\times n}$ as vectors in $\R^{n^2}$.
Notice that $\rho^{p\beta}$ belongs to $L^1(\Omega)$ thus 
\[L^q(\Omega,\Rnn,\rho^{-\beta})\subset L^1(\O,\Rnn).\] 
Moreover, using that $\Omega$ is bounded it follows that $\V\subset L^\infty(\Omega,\Rnn)$. Hence, $\W$ is well-defined.

\begin{lemma}\label{sum} The space $L^q(\Omega,\R^{n\times n},\rho^{-\beta})$ can be written as $\W\oplus\rho^{p\beta}\V.$ Moreover, for all $\tF=\tg+\rho^{p\beta}\tpsi$ in $\W\oplus\rho^{p\beta}\V$ it follows that $\|\tg\|_{L^q(\Omega,\rho^{-\beta})}\leq C_2\|\tF\|_{L^q(\O,\rho^{-\beta})}$ where
\begin{align}\label{dense cons}
C_1= \left(1+\sum_{j=1}^m \|\tpsi_j\|_{L^p(\O,\rho^\beta)} \|\tpsi_j\|_{L^q(\O,\rho^\beta)}\right).
\end{align}
The collection $\{\tpsi_j\}_{1\leq j\leq m}$ in the previous identity is an arbitrary orthonormal basis of $\V$ with respect to the inner product
\begin{equation*}
\langle \tpsi,\tvfi\rangle_\O=\int_\O \tpsi(x):\tvfi(x) \, \rho^{p\beta}(x)\,\d x.
\end{equation*}
\end{lemma}
\begin{proof} Notice that $\rho^{p\beta}\tpsi$, with $\tpsi\in\V$, belongs to $L^q(\Omega,\Rnn,\rho^{-\beta})$. Indeed, 
\[\|\rho^{p\beta}\tpsi\|^q_{L^q(\O,\rho^{-\beta})} = \int_\O\|\tpsi\|_q^q\rho^{pq\beta}\o^{-q\beta}  \leq \sup_{x\in\Omega}\|\tpsi(x)\|_q^q \int_\Omega \rho^{p\beta}(x)\dx.\]
Thus, $\rho^{p\beta}\V$ is a subspace of $L^q(\O,\Rnn,\rho^{-\beta})$.
  
The representation follows naturally from the definition of $\W$. Indeed, given $\tF$ in the space $L^q(\Omega,\R^{n\times n},\rho^{-\beta})$ we take
\begin{align}\label{rep}
\tpsi_{\tF}(x):=\sum_{j=1}^m \alpha_{\tF,j}\tpsi_j(x),
\end{align}
where $\alpha_{\tF,j}:=\int_\O \tF : \tpsi_j$ for any $1\leq j\leq m$. Thus, $\tF=\th_{\tF}+\rho^{p\beta}\tpsi_{\tF}$ where $\tpsi_{\tF} \in \V$  and  $\th_{\tF}:=\tF-\rho^{p\beta}\tpsi_{\tF}\in \W$. The uniqueness is a simple exercise of linear algebra. Now, to obtain (\ref{dense cons}) notice that the coefficients  $\alpha_{\tF,j}$ verify
\begin{align}\label{rep coeff}
|\alpha_{\tF,j}|\leq \|\tF\|_{L^q(\Omega,\rho^{-\beta})}\|\tpsi_j\|_{L^p(\Omega,\rho^\beta)},
\end{align}
for all $j$. Thus, from (\ref{rep}) and (\ref{rep coeff}) we have
\begin{align*}
\|\th_{\tF}\|_{L^q(\Omega,\rho^{-\beta})}&\leq \left(1+\sum_{j=1}^m \|\tpsi_j\|_{L^p(\O,\rho^\beta)} \|\rho^{p\beta}\tpsi_j\|_{L^q(\O,\rho^{-\beta})}\right)\|\tF\|_{L^q(\O,\rho^{-\beta})}\\
&= \left(1+\sum_{j=1}^m \|\tpsi_j\|_{L^p(\O,\rho^\beta)} \|\tpsi_j\|_{L^q(\O,\rho^\beta)}\right)\|\tF\|_{L^q(\O,\rho^{-\beta})}.
\end{align*}
\end{proof}

Given a decomposition of extended Whitney cubes $\{\O_t\}_{t\in\Gamma}$ of $\O$, we define the subspace $\W_0\subset \W$ by: 
\begin{equation*}
\W_0:=\{\tg\in \W: supp(\tg)\cap \O_t\neq \emptyset\text{ only for a finite number of }t\in\Gamma\}.
\end{equation*}

\begin{lemma}\label{DecompSimple} Let $\Omega\subset\R^n$ be a bounded John domain and $\{\O_t\}_{t\in \Gamma}$ a decomposition of extended Whitney cubes. Then, there exists a positive constant $C_0$ such that for any $\tg\in\W_0$, there is a collection of functions $\{\tg_t\}_{t\in\Gamma}$ with the following properties:
\begin{enumerate}
\item $\tg=\sum_{t\in \Gamma} \tg_t.$
\item $supp(\tg_t)\subset\Omega_t.$
\item $\tg_t\in\W_0$, for all $t\in\Gamma$.
\end{enumerate}
We call this collection of functions a $\V$-decomposition of $\tg$ subordinate to $\{\Omega_t\}_{t\in\Gamma}$. 

In addition, it satisfies 
\begin{align}\label{Estimation}
\sum_{t\in \Gamma} \|\tg_t\|^q_{L^q(\O_t,\rho^{-\beta})}\leq C^q_0\|\tg\|^q_{L^q(\O,\rho^{-\beta})}.
\end{align}

Moreover, $\tg_t\not\equiv 0$ only for a finite number of $t\in\Gamma$.
\end{lemma}

Lemma \ref{DecompSimple}, which is fundamental in these notes, is proved in the next section.

Now, we define the following subspace $\S$ of $L^q(\Omega,\R^{n\times n},\rho^{-\beta})$ by
\begin{equation}\label{S}
\S:=\W_0\oplus\rho^{p\beta}\V.
\end{equation}

\begin{lemma}\label{dense} 
The subspace $\S$ defined above is dense in $L^q(\Omega,\R^{n\times n},\rho^{-\beta})$. 
\end{lemma}

\begin{proof} By Lemma \ref{sum}, it is sufficient to show that $\W_0$ is dense in $\W$ with respect to the norm in $L^q(\Omega,\R^{n\times n},\rho^{-\beta})$. 

Let $Q\subset\Omega$ be a cube that intersects a finite collection of $\Omega_t$ and $\{\tnu_j\}_{1\leq j\leq m}$ an orthogonal basis of the finite dimensional space $\V$ with respect to the inner product 
\[\langle \tpsi,\tvfi\rangle_Q=\int_Q \tpsi(x) : \tvfi(x)\, \rho^{p\beta}(x)\,\d x.\] 
Now, given $\th\in\W$ notice that $\int_\Omega \th : \tnu_j=0$ for all $j$, for being $\tilde{\nu}_j$ a function in $\V$. Next, given $\epsilon>0$, let $\Omega_\epsilon\subset \Omega$ be
an open set that contains $Q$, intersects a finite number of $\Omega_t$ and 
\[\|(1-\chi_{\Omega_\epsilon})\th\|_{L^q(\O,\rho^{-\beta})}<\epsilon.\]
Thus, we define the function $\tg$ by
\begin{align*}
\tg(x):=\chi_{\Omega_\epsilon}(x)\th(x)+\sum_{i=1}^m\chi_Q(x)\rho^{p\beta}(x) \tnu_i(x)\int_{\Omega\setminus\Omega_\epsilon}\th(y) : \tnu_i(y)\,\d y.
\end{align*}
Notice first that the support of $\tg$ intersects a finite number of $\O_t$, and $\int_\Omega \tg : \tnu_j=0$ for all $j$, thus $\tg$ belongs to $\W_0$. Moreover,
\begin{eqnarray*}
\|\th-\tg\|_{L^q(\Omega,\rho^{-\beta})}&\leq& \epsilon+\sum_{i=1}^m\left\|\chi_Q(x)\rho^{p\beta}(x)\tnu_i(x)\int_{\Omega\setminus\Omega_\epsilon}\th(y) : \tnu_i(y)\,\d y\right\|_{L^q(\Omega,\rho^{-\beta})}\\
&\leq& \epsilon+\sum_{i=1}^m\left(\int_{\Omega\setminus\Omega_\epsilon}|\th(y) : \tnu_i(y)|\,\d y\right)\|\tnu_i\rho^{p\beta}\|_{L^q(Q,\rho^{-\beta})}\\
&\leq& \epsilon+\sum_{i=1}^m \|\th\|_{L^q({\Omega\setminus\Omega_\epsilon},\rho^{-\beta})}  \|\tnu_i\|_{L^p(\Omega,\rho^\beta)}    \|\tnu_i\rho^{p\beta}\|_{L^q(Q,\rho^{-\beta})} \\
&\leq& \epsilon\left(1+\sum_{i=1}^m\|\tnu_i\|_{L^p(\Omega,\rho^\beta)}    \|\tnu_i\rho^{p\beta}\|_{L^q(Q,\rho^{-\beta})}  \right),
\end{eqnarray*}
which proves that $\S$ is dense in $L^q(\Omega,\Rnn,\rho^{-\beta})$.
\end{proof}

We are ready to prove the main result of this article.

\begin{proof} [Proof of Theorem \ref{gkJohn}] Let $\vv$ be an arbitrary vector field in $W^{1,p}(\O,\R^n,\rho^\beta)$. Next, let us take $\ww$ in $\Sigma$, the kernel of $l$, such that   
\[\int_\O (D\vv:\tvfi)\,\rho^{p\beta}=\int_\O (D\ww:\tvfi)\,\rho^{p\beta}\]
for all $\tvfi\in \V$. Recall that the elements in $\V$ are the differential matrix of the vector fields in (\ref{Kernel}), thus $D\ww$ is the orthogonal projection of $D\uu$ on $\V$ with respect to the inner product used above. Moreover, $\beta\geq 0$ and $\O$ is bounded, then $\ww$ belongs to $W^{1,p}(\O,\R^n,\rho^\beta)$.  Hence, by taking $\uu:=\vv-\ww$, it is sufficient to prove 
\begin{align*}
\left(\int_\Omega |D\uu|^p\rho^{p\beta}\, \d x\right)^{1/p} \leq C \left(\int_\Omega |l(\uu)|^p\rho^{p\beta}\, \d x\right)^{1/p},
\end{align*}
for all $\uu\in W^{1,p}(\O,\R^n,\rho^\beta)$ which satisfy
\[\int_\O (D\uu:\tvfi)\,\rho^{p\beta}=0\]
for all $\tvfi\in \V$. For simplicity, we preferred to write the generalized version of Korn inequality in our main theorem by using the infimum over $l(\ww)=0$, however, it is also valid for vector fields verifying the condition stated above, which is also very standard in this kind of inequalities.

Now, using that the space $\S$ defined in (\ref{S}) is dense in $L^q(\O,\Rnn,\rho^{-\beta})$, it is enough to  show that 
there is a constant $C$ such that 
\begin{align*}
\int D\uu : (\tg+\rho^{p\beta}\tpsi) \leq C \left(\int_\Omega |l(\uu)|^p\rho^{p\beta}\, \d x\right)^{1/p},
\end{align*}
for any arbitrary function $\tg+\rho^{p\beta} \tpsi$ in $\S$, with $\|\tg+\rho^{p\beta}\tpsi\|_{L^q(\O,\rho^{-\beta})}\leq 1$. Thus, given a function $\tg+\rho^{p\beta} \tpsi$ with norm less than one, let us take a $\V$-decomposition $\{\tg_t\}_{t\in\Gamma}$ of $\tg$ (see Lemma \ref{DecompSimple}). Thus, 
\begin{align*}
\int_{\Omega} D\uu:(\tg+\rho^{p\beta}\tpsi)&=\int_{\Omega} D\uu:\tg\\
&=\int_{\Omega_t} D\uu:\sum_{t\in\Gamma} \tg_t\\
&=\sum_{t\in\Gamma}\int_{\Omega_t} D\uu: \tg_t=(1).
\end{align*}
Notice that in the last identity we used the finiteness of the sum stated in Lemma \ref{DecompSimple}. Now, $\tg_t$ satisfies that $\int_{\Omega_t} D\ww:\tg_t=0$ for all $\ww\in\Sigma$. Thus, from H\"older inequality, property $(2)$ in Whitney decomposition's definition and (\ref{infwKorn}) we obtain
\begin{align*}
(1)&\leq \sum_{t\in\Gamma} \inf_{l(\ww)=0}\|D\uu-D\ww\|_{L^p(\O_t,\rho^\beta)}\|\tg_t\|_{L^{q}(\O_t,\rho^{-\beta})}\\
&\leq \sum_{t\in\Gamma} C(\diam(\O_t))^\beta \inf_{l(\ww)=0}\|D\uu-D\ww\|_{L^p(\O_t)}\|\tg_t\|_{L^{q}(\O_t,\rho^{-\beta})}\\
&\leq C \sum_{t\in\Gamma} (\diam(Q_t))^\beta \|l(\uu)\|_{L^p(\O_t)}\|\tg_t\|_{L^{q}(\O_t,\rho^{-\beta})}\\
&\leq C \sum_{t\in\Gamma} \|l(\uu)\|_{L^p(\O_t,\rho^\beta)}\|\tg_t\|_{L^{q}(\O_t,\o^{-\beta})}=(2).
\end{align*}
Next, we use H\"older inequality for the sum depending on $t$ to obtain
\begin{align*}
(2)&\leq C \left(\sum_{t\in\Gamma} \int_{\Omega_t} | l(\uu)|^p\rho^{p\beta}\right)^{1/p}\left(\sum_{t\in\Gamma} \|\tg_t\|_{L^{q}(\O_t,\rho^{-\beta})}^q\right)^{1/{q}}=(3).
\end{align*} 
Using that each cube $\O_t$ intersects no more than $12^n$ cubes in $\{\O_s\}_{s\in\Gamma}$, and Lemma \ref{DecompSimple} we conclude
\begin{align*}
(3)&\leq C\,C_0\,\|l(\uu)\|_{L^p(\O,\rho^\beta)} \|\tg\|_{L^q(\O,\rho^{-\beta})}\leq C\,C_0\, C_1 \|l(\uu)\|_{L^p(\O,\o)}, 
\end{align*}  where $C$ is independent of $\O$,  $C_0$ is the estimate in (\ref{Estimation}) and $C_1$ is the constant in (\ref{dense cons}). 
\end{proof}

In the following two corollaries we generalize the Korn type inequalities proved in \cite{R1,Da} to bounded John domains.

\begin{corollary}\label{ReshonJohn} Let $\O\subset\R^n$ be a bounded John domain, with $n\geq 3$, $1<p<\infty$ and  $\Pi:W^{1,p}(\O,\R^n)\to \Sigma$ a projection, where $\Sigma$ is endowed with the topology of $L^p(\O,\R^n)$. Then, there is a constant $C$ such that 
\begin{equation*}
\|\vv-\Pi(\vv)\|_{W^{1,p}(\O)}\leq C\|l(\vv)\|_{L^p(\O)},
\end{equation*}
for all $\vv\in W^{1,p}(\O,\R^n)$.
\end{corollary}
\begin{proof} Let $\ww_0\in\Sigma$ be such that
\begin{align*}
\|D\vv-D\ww_0\|_{L^p(\Omega)}\leq 2 \inf_{l(\ww)=0}\|D\vv-D\ww\|_{L^p(\Omega)}
\end{align*}
and $\int_\Omega \vv_i-{\ww_0}_i=0$ for all $1\leq i\leq n$.
Then, by using that $\Pi$ is a projection and the norms $\|\cdot \|_{W^{1,p}(\O)}$ or $\|\cdot\|_{L^p(\O)}$ are equivalent over $\Sigma$ for being $\Sigma$ a finite dimensional space, we have
\begin{align*}
\|\vv-\Pi(\vv)\|_{W^{1,p}(\Omega)}&\leq \|\vv-\ww_0\|_{W^{1,p}(\Omega)}+\|\Pi(\vv)-\Pi(\ww_0)\|_{W^{1,p}(\Omega)}\\
&\leq C\{\|\vv-\ww_0\|_{W^{1,p}(\Omega)}+\|\Pi(\vv)-\Pi(\ww_0)\|_{L^p(\Omega)}\}\\
&\leq C\|\vv-\ww_0\|_{W^{1,p}(\Omega)}.
\end{align*}
Finally, by using Poincar\'e inequality on $\O$ (see for example \cite{M}) and Theorem \ref{gkJohn} with $\beta=0$ we conclude 
\begin{align*}
\|\vv-\Pi(\vv)\|_{W^{1,p}(\Omega)}&\leq C\|D\vv-D\ww_0\|_{L^p(\Omega)}\\
&\leq C\inf_{l(\ww)=0}\|D\vv-D\ww\|_{L^p(\Omega)}\\
&\leq C\|l(\vv)\|_{L^p(\O)}.
\end{align*}

\end{proof}

\begin{corollary}\label{DainhonJohn} Let $\O\subset\R^n$ be a bounded John domain, with $n\geq 3$. Then, there is a constant $C$ such that 
\begin{equation}\label{DainConsJohn}
\|\vv\|_{W^{1,2}(\O)}\leq C\{\|\vv\|_{L^2(\O)}+\|l(\vv)\|_{L^2(\O)}\},
\end{equation}
for all $\vv\in W^{1,2}(\O,\R^n)$.
\end{corollary}
\begin{proof} Let $Q$ be a cube included in $\O$ and $\Pi_Q:W^{1,2}(Q,\R^n)\to \Sigma$ a projection. The norms $\|\cdot \|_{L^2(Q)}$ and $\|\cdot\|_{L^2(\O)}$ are equivalent over $\Sigma$, thus $\Pi_Q$ is also a projection from $W^{1,2}(\O,\R^n)$ to $\Sigma$. 
Hence, using Corollary \ref{ReshonJohn} over $\O$ and the validity of (\ref{DainConsJohn}) over the cube $Q$, we show for any $\vv$ in $W^{1,2}(\O,\R^n)$ that 
\begin{align*}
\|\vv\|_{W^{1,2}(\Omega)}&\leq \|\vv-\Pi_Q(\vv)\|_{W^{1,2}(\Omega)}+\|\Pi_Q(\vv)\|_{W^{1,2}(\Omega)}\\
&\leq C\{\|l(\vv)\|_{L^2(\Omega)}+\|\Pi_Q(\vv)\|_{L^2(Q)}\}\\
&\leq C\{\|l(\vv)\|_{L^2(\Omega)}+\|\vv\|_{W^{1,2}(Q)}\}\\
&\leq C\{\|l(\vv)\|_{L^2(\Omega)}+\|\vv\|_{L^2(Q)}+\|l(\vv)\|_{L^2(Q)}\}\\
&\leq C\{\|\vv\|_{L^2(Q)}+\|l(\vv)\|_{L^2(\Omega)}\}
\end{align*}
concluding the proof.
\end{proof}

To finalize this section we show an example that proves that the generalized Korn inequality (\ref{wgkJohn}) might fail if $\O$ is not a John domain. Thus, let us consider the case $p=2$, $\beta=0$ and the cuspidal domain $\O\subset\R^3$ given by  
\[\O:=\{(x_1,x_2,x_3)\in\R^3\,:\, 0<x_1,x_2<1\text{ and }0<x_3<x_2^\gamma\},\]
with $\gamma>1$. Let us assume by contradiction that (\ref{wgkJohn}) holds on $\O$, thus, following the ideas in the previous two corollaries, we can conclude that there is a constant $C$ such that 
\begin{align}\label{Counter}
\|D\vv\|_{L^2(\O)}\leq C\{\|\vv\|_{L^2(Q)}+\|l(\vv)\|_{L^2(\O)}\},
\end{align}
for all $\vv\in W^{1,2}(\O,\R^3)$, where $Q$ is a fixed cube in $\O$. Now, let us consider the vector field 
\[\vv(x_1,x_2,x_3):=(0,\,-(s+1)x_3x_2^s,\,x_2^{s+1}).\] 
By a straightforward calculation it can be seen that if $s$ satisfies that 
\begin{align*}
\max\left\{-\dfrac{(\gamma+1)}{2}-(\gamma-1),-\dfrac{(\gamma+1)}{2}-1\right\} < s <-\dfrac{(\gamma+1)}{2}
\end{align*}
then the left hand side of (\ref{Counter}) is infinite while the right one is finite following in a contradiction. 

This kind of counterexamples has been studied in \cite{ADL} to show that the Korn inequality \begin{align}\label{Korn}
\|D\vv\|_{L^2(\O)}\leq C\{\|\vv\|_{L^2(\O)}+\|\varepsilon(\vv)\|_{L^2(\O)}\},
\end{align}
fails on certain cuspidal domains of the style of $\O$. The fact that this counterexample also works for the generalized version of Korn can also be concluded by observing that 
\begin{align*}
\|l(\vv)\|_{L^2(\O)}^2=\|\varepsilon(\vv)\|_{L^2(\O)}^2-\frac{1}{n}\|{\rm div}\,\vv\|^2_{L^2(\O)}\leq \|\varepsilon(\vv)\|_{L^2(\O)}^2
\end{align*}
which implies that (\ref{Counter}) fails to be true when (\ref{Korn}) does.

\section{$\V$-decomposition of functions}
\label{V-decomp}
\setcounter{equation}{0}

The $\V$-decomposition of functions introduced in Lemma \ref{DecompSimple} is constructed by using an inductive argument based on a certain partial order on the Whitney cubes $\{Q_t\}_{t\in\Gamma}$.

Let us denote by $G=(V,E)$ a graph with vertices $V$ and edges $E$. Graphs in these notes do not have neither multiple edges nor loops and the number of vertices in $V$ is countable. Moreover, each vertex is of finite degree, i.e. only a finite number of vertices emanate from each vertex. A {\it rooted tree} (or simply a tree) is a connected graph $G=(V,E)$ in which any two vertices are connected by exactly one simple path, and a {\it root} is simply a distinguished vertex $a\in V$. Moreover, if $G=(V,E)$ is a rooted tree, it is possible to define a {\it partial order} ``$\preceq$" in $V$ as follows: $s\preceq t$ if and only if the unique path connecting $t$ with the root $a$ passes through $s$. The {\it height} or {\it level} of any $t\in V$ is the number of vertices in $\{s\in V\,:\,s\preceq t\text{ with }s\neq t\}$. {\it The parent} of a vertex $t\in V$ is the vertex $s$ satisfying that $s\preceq t$ and its height is one unit smaller than the height of $t$. We denote the parent of $t$ by $t_p$. 
It can be seen that each $t\in V$ different from the root has a unique parent, but several elements in $V$ could have the same parent. Note that two vertices are connected by an edge ({\it adjacent vertices}) if one is the parent of the other.

Now, if $\O\subset\R^n$ is a bounded domain and $\{Q_t\}_{t\in\Gamma}$ a Whitney decomposition, we define the connected graph 
\begin{align*}
G_\Gamma=(\Gamma,E_\Gamma)
\end{align*} 
where the set of vertices is the set of subindexes $\Gamma$ and two arbitrary $s,t$ in $\Gamma$ are connected by an edge iff $Q_s$ and $Q_t$ are $(n-1)$- neighbors. 

\begin{defi}\label{TreeStructure} A {\it tree structure} of $\Gamma $ is given by a collection of edges $E\subset E_\Gamma$ and a distinguished vertex $a\in\Gamma$ such that the subgraph $G=(\Gamma,E)$ of $G_\Gamma$ is a rooted tree.
\end{defi}

There are different tree structures that can be added to $\Gamma$. For example, we can define one such that the path that connects each vertex $t$ with the root $a$ has minimal length. This kind of paths are not unique thus the tree structure must be defined inductively, level by level, by choosing a path with minimal length. This example was considered in \cite{L1} and it can be done for any arbitrary proper domain in $\R^n$ without any assumption on the geometry. 

In the following picture we sketch another example that shows some cubes of a Whitney decomposition of a circle and how a tree structure looks like. 
\begin{figure}[htb]
       \center{\includegraphics[width=70mm]{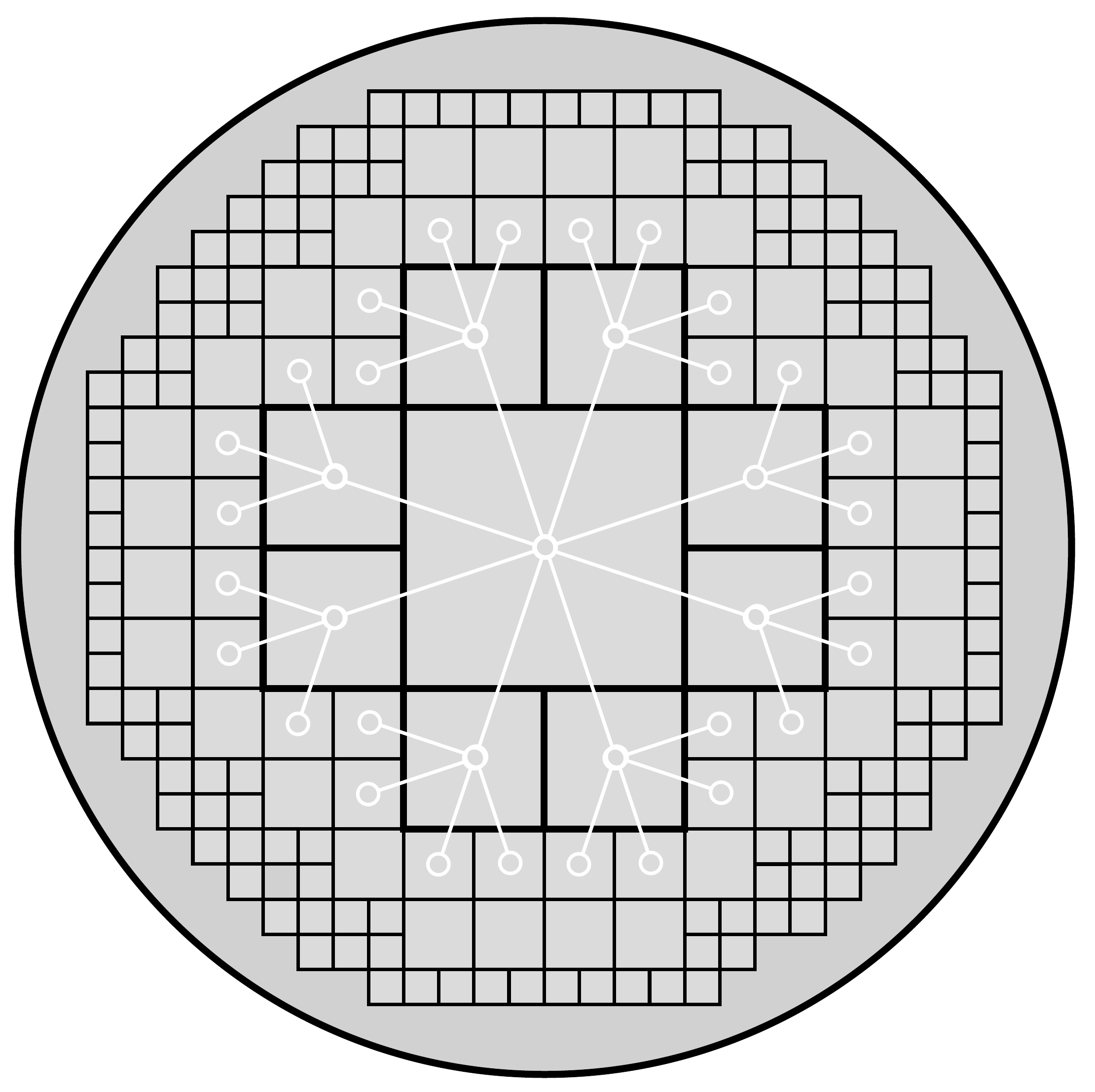}}
        \caption{ Some Whitney cubes $Q_t$ of a circle and a tree structure}
\end{figure}

\begin{defi}\label{OperT} Given a Whitney decomposition $\{Q_t\}_{t\in\Gamma}$ of a domain $\Omega\subset\R^n$ and a tree structure of $\Gamma$, we denote by $\{B_t\}_{t\neq a}$ a collection of open pairwise disjoint cubes with sides parallel to the axis such that $B_t\subseteq\O_{t}\cap\O_{t_p}$ and $|\O_t|\leq C_n |B_t|$ for all $t\in\Gamma$. This collection of cubes exists by following the properties for Whitney cubes described in Section \ref{Prel}. Thus, using the tree structure of $\Gamma$, we define {\it the Hardy type operator} $T$ for functions in $L^1(\O)$ by:
\begin{align}\label{Ttree}
Tg(x):=\sum_{a\neq t\in\Gamma}\dfrac{\chi_{t}(x)}{|W_t|}\int_{W_t}|g|,
\end{align} 
where $\{\O_t\}_{t\in\Gamma}$ is the collection of extended Whitney cubes associated to $\{Q_t\}_{t\in\Gamma}$,  $\chi_t(x)$ is the characteristic function of $B_t$, for all $t\neq a$, and $\displaystyle{W_t=\bigcup_{s\succeq t} \O_s}$.
\end{defi}
We refer to $W_t$ by {\it the shadow of }$\Omega_t$. This is a fairly known name and it follows the assumption that light travels from $\O_a$ to the different cubes $\O_t$ along the unique path that connects them. This geometric interpretation was taken from \cite{SS}, page 81, in the context of quasi-hyperbolic geodesics and chains of Whitney cubes with minimal number of cubes.

Now, if $\O$ is a John domain and $\{Q_t\}_{t\in\Gamma}$ is a Whitney decomposition of $\O$, it is possible to choose a tree structure for $\Gamma$ which satisfies a certain geometric property. See the following lemma which has been proved in \cite{L2}.

\begin{lemma}\label{tcovJohn} Given a bounded John domain $\O$ and a Whitney decomposition $\{Q_t\}_{t\in\Gamma}$, there exists a constant $K>1$ and a tree structure for $\Gamma$, with root ``$a$", that satisfies 
\begin{align}\label{Boman tree}
Q_s\subseteq KQ_t,
\end{align}
for any $s,t\in\Gamma$ with $s\succeq t$. In other words, the shadow of $Q_t$ is contained in $KQ_t$. 
\end{lemma}

From now on $\Omega\subset\R^n$ is a bounded John domain, $\{\O_t\}_{t\in\Gamma}$ is the collection of extended Whitney cubes defined in Section \ref{Prel} and $\Gamma$ has a tree structure with the geometric property introduced in Lemma \ref{tcovJohn}.

\begin{lemma}\label{weighted T}
Let $\O\subset\R^n$ be a bounded John domain, $\beta\geq 0$ and $1<q<\infty$. Then, the operator $T$ defined in (\ref{Ttree}) is continuous from $L^q(\O,\rho^{-\beta})$ to itself, where $\rho$ is the distance to the boundary of $\O$. Moreover, its norm is bounded by 
\[\|T\|_{L\to L} \leq C K^\beta,\]
where $L$ denotes $L^q(\O,\rho^{-\beta})$. The constant $C$ in the previous inequality is independent of $\O$ and $K$ is the constant introduced in (\ref{Boman tree}).
\end{lemma}
This result was proved in \cite{L2}.

Now, we are ready to construct the $\V$-decomposition.

\begin{lemma}\label{Decomp} Let $\Omega\subset\R^n$ be a bounded John domain and $\{\O_t\}_{t\in \Gamma}$ a collection of extended Whitney cubes. Given $\tg\in L^1(\Omega,\R^{n\times n})$ such that $\int_\O \tg:\tvfi=0$, for all $\tvfi\in \V$, and $supp(\tg)\cap \O_s\neq \emptyset$ only for a finite number of $s\in\Gamma$, there exists a collection of functions $\{\tg_t\}_{t\in\Gamma}$  in $L^1(\Omega,\Rnn)$ with the following properties:
\begin{enumerate}
\item $\tg=\sum_{t\in \Gamma} \tg_t.$
\item $supp(\tg_t)\subset\Omega_t.$
\item $\int_{\Omega_t} \tg_t:\tvfi=0$, for all $\tvfi\in \V.$
\end{enumerate}

Moreover, let $t\in\Gamma$. If $x\in B_s$, where $s=t$ or $s_p=t$, we have the following pointwise estimate 
\begin{align}\label{P12}
\|\tg_t\|_\infty(x)\leq \|\tg\|_\infty(x)+ C_n K^{n+1}T\|\tg\|_1(x),
\end{align}
where $K$ is the geometric constant introduced in (\ref{Boman tree}) and $C_n$ is a constant that depends only on $n$. Otherwise, if $x\not\in\bigcup_{s\in\Gamma} B_s$ or $x\in B_s$, where $s\neq t$ and $s_p\neq t$, then
\begin{align}\label{P11}
\|\tg_t\|_\infty(x)\leq \|\tg\|_\infty(x).
\end{align}
Finally, $\tg_t \equiv 0$ for all $t\in \Gamma\setminus\Gamma_{\tg}$, where $\Gamma_{\tg}$ is the subtree of $\Gamma$ with a finite number of vertices given by 
\begin{align*}
\Gamma_{\tg}:=\{s\in\Gamma\,:\,\ \text{there is }k\succeq s\text{ with }supp(\tg)\cap \O_k\neq\emptyset\}.
\end{align*}

\end{lemma}
\begin{proof} Let us define a base of the vector space $\V$. For the constant skew-symmetric matrices we consider
\begin{equation*}
(E_{ij})_{i'j'}=\left\{
  \begin{array}{r l}
     1 & \quad \text{if }(i',j')=(i,j)\\
     -1 & \quad \text{if }(i',j')=(j,i)\\
     0 & \quad \text{otherwise}\\
   \end{array} \right.
\end{equation*}
where $1\leq i<j\leq n$. It can  be seen that the dimension of $\V$ equals $n_0+n$, where $n_0:=\frac{n(n-1)}{2}+1$. Let us take the following basis  $\{A_1,\cdots,A_{n_0},H_1(x),\cdots,H_n(x)\}$ of $\V$, where the first $\frac{n(n-1)}{2}$ elements are the matrices with constant coefficients $E_{ij}$, following an arbitrary order, and $A_{n_0}=I$ is the identity matrix. The matrices $H_i(x)$ have been previously defined in (\ref{MatrHi}).

Now, let $\{\phi_t\}_{t\in\Gamma}$ be a partition of the unity subordinate to $\{\O_t\}_{t\in\Gamma}$. Namely, a collection of smooth functions such that $\sum_{t\in\Gamma} \phi_t=1$, $0\leq \phi_t\leq1$ and $supp(\phi_t)\subset\Omega_t$. Thus, $\tg$ can be cut-off into $\tg=\sum_{t\in\Gamma} \tf_t$ by taking $\tf_t(x)=\phi(x)\tg(x)$. Note that $\tf_t\equiv 0$ except for a finite number of $t\in\Gamma$. This decomposition verifies (1) and (2) in the statement of the lemma but probably it does not satisfy (3). Thus, we make some modifications to obtain the orthogonality with respect to $\V$. We construct the decomposition in two steps. We deal first with the orthogonality with respect to the matrices with constant coefficients $\{A_1,\cdots, A_{n_0}\}$ and later with  respect to $\{H_1(x),\cdots,H_n(x)\}$. 

{\it First step}: The decomposition in this first step is denoted with the upper index (0). Thus, we define the functions $A_{i,s}(x)$ as a sort of normalization of $A_i$ with respect to a certain inner product over $B_s$: 
\begin{equation*}
\begin{split}
A_{i,s}(x)&:=\frac{\chi_s(x)A_i}{2|B_s|}\hspace{1cm} \text{if }1\leq i\leq n_0-1\\
A_{n_0,s}(x)&:=\frac{\chi_s(x)A_{n_0}}{n|B_s|}=\frac{\chi_s(x)I}{n|B_s|},
\end{split}
\end{equation*}
where $\chi_s(x)$ is the characteristic function of $B_s$. Indeed, notice that $\int A_{i,s}(x):A_{j}\dx=\delta_{i,j}$ for all $s\in\Gamma\setminus\{a\}$ and  $1\leq i, j\leq n_0$, where $\delta_{i,j}$ is the Kronecker delta. 

Thus, we define the collection of functions $\{\tg^{(0)}_t\}_{t\in\Gamma}$ from $\Omega$ to $\R^{n\times n}$ by
\begin{align}\label{A-decomp}
\tg^{(0)}_t(x):=\tilde{f}_t(x)+\left(\sum_{s:\,s_p=t}\th^{(0)}_s(x)\right)-\th^{(0)}_t(x), 
\end{align}
where 
\begin{align}\label{h0}
\th^{(0)}_s(x):=\sum_{i=1}^{n_0}\left(\int_{W_s}A_i:\sum_{k\succeq s}\tf_k(y)\dy\right) A_{i,s}(x). 
\end{align}
The sum in (\ref{A-decomp}) is indexed over every $s\in\Gamma$ such that $t$ is the parent of $s$. In the particular case when $t$ is the root of $\Gamma$,  (\ref{A-decomp}) means 
\[\tg^{(0)}_a(x)=\tf_a(x)+\sum_{s:\,s_p=a}\th^{(0)}_s(x).\]

Notice that the functions $\th^{(0)}_s$ in (\ref{h0}) are well-defined because of the integrability of $\tg$. Indeed, 
\begin{align*}
\left|\int_{W_s}A_i:\sum_{k\succeq s}\tf_k(y)\dy\right|&\leq \|A_i\|_\infty\int_{W_s}\|\sum_{k\succeq s}\tf_k\|_1(y)\dy\\
&\leq \int_{W_s}\|\tg\|_1(y)\dy.
\end{align*}
See definitions of $\|\tg\|_r(y)$, for $1\leq r\leq \infty$, in Section \ref{Prel}.
Moreover, it can be easily observed that $\tf_s$, $\thc_s$ and $\tgc_s$ are identically zero for all $s\in\Gamma\setminus\Gamma_{\tg}$.
For this reason the sums indexed over subsets of $\Gamma$, for instance $k\succeq s$, that appear in the first step are finite. The finiteness of these sums is also verified in the second step.

We know that
\[supp(\th^{(0)}_s)\subset B_s\]
and the coefficients of $\th_s^{(0)}(x)$ are estimated in the following way:
\begin{align}\label{esths}
\|\th^{(0)}_s\|_\infty(x) \leq \tfrac{1}{2|B_s|} \chi_s(x)\int_{W_s}\|\tg\|_1(y)\dy=\tfrac{|W_s|}{2|B_s|} \chi_s(x) T\|\tg\|_1(x)
\end{align}
for all $x\in\Omega$. Thus, using (\ref{Boman tree}), if $x\in B_s$ where $s=t$ or $s_p=t$ then 
\begin{equation}\label{P02}
\begin{split}
\|\tg_t^{(0)}\|_\infty(x)&\leq \|\tg\|_\infty(x)+\tfrac{|W_s|}{2|B_s|}T\|\tg\|_1(x)\\
&\leq \|\tg\|_\infty(x)+K^n\,T\|\tg\|_1(x).
\end{split}
\end{equation}
Otherwise, if $x\not\in\bigcup_{s\in\Gamma} B_s$ or $x\in B_s$ with $s\neq t$ and $s_p\neq t$ then
\begin{align*}
\|\tg_t^{(0)}\|_\infty(x)\leq \|\tg\|_\infty(x).
\end{align*}

Let us continue by showing that $\tg(x)=\sum_{t\in\Gamma} \tg^{(0)}_t(x)$ for all $x$. Given $x\in \Omega$, let suppose that $x\notin \bigcup B_t$. Then $\tg_t^{(0)}(x)=\tg_t(x)$ for all $t\in\Gamma$, then \[\sum_{t\in\Gamma} \tg^{(0)}_t(x)=\sum_{t\in\Gamma} \tg_t(x)=\tg(x).\]
Otherwise, if $x$ belongs to $B_{k_0}$ for ${k_0}\in\Gamma$ it follows that $\tg_t^{(0)}(x)=\tg_t(x)$ for all $t\neq {k_0}$,  $t\neq {k_0}_p$ (${k_0}_p$ is the parent of $k_0$). Moreover, by using that the cubes $B_s$ are pairwise disjoint we have
 \begin{align*}
\tg^{(0)}_{{k_0}}(x)&=\tg_{k_0}(x)-\th^{(0)}_{k_0}(x)\\
\tg^{(0)}_{{k_0}_p}(x)&=\tg_{{k_0}_p}(x)+\th^{(0)}_{k_0}(x).
\end{align*}
Then, $\sum_{t\in\Gamma} \tg^{(0)}_t(x)=\tg(x)$ for all $x$.

Next, let us prove (2) in the statement of the lemma. The parent of each $s$ in the sum in (\ref{A-decomp}) is $t$, then $B_s\subseteq \Omega_s\cap\Omega_t$. Thus,  $supp(\tg^{(0)}_t)\subseteq\Omega_t$.

Now, let us show property (3), which refers to the orthogonality of $\tg^{(0)}_t$ with respect to the matrices $A_1,\cdots,A_{n_0}$ for all $t\in\Gamma$. Observe that $k\succeq t$ if and only if $k\succeq s$, with $s_p=t$, or $k=t$. Thus, given $1\leq j\leq n_0$
\begin{align*}
\int \th^{(0)}_t(x):A_j\dx&=\int A_j:\sum_{k\succeq t}\tf_k(y)\dy\\
&=\int A_j:\tf_t(y)\dy\ +\sum_{s:\,s_p=t}\int A_j:\sum_{k\succeq s}\tf_k(y)\dy\\
&=\int A_j:\tf_t(y)\dy\ +\sum_{s:\,s_p=t}\int \th^{(0)}_s(x):A_j\dx.
\end{align*}
Then, $\int \tg_t^{(0)}(x):A_j\dx=0$, for all $t\neq a$. 

Finally, 
\begin{align*}
\int \tg^{(0)}_a(x):A_j\dx&=\int \tf_a(x):A_j\dx+\sum_{s:\,s_p=a}\int A_j:\sum_{k\succeq s}\tf_k(y)\dy\\
&=\int A_j:\sum_{k\succeq a}\tf_k(y)\dy\\
&=\int A_j : \tg (x)\dx = 0.
\end{align*}

\bs

{\it Second step}: In this step the decomposition is denoted with the upper index (1). Now, we repeat the process used in the first step replacing the collection $\{\tf_t\}_{t\in\Gamma}$ by $\{\tg^{(0)}_t\}_{t\in\Gamma}$ and the matrices $A_1,\cdots,A_{n_0}$ by $H_1(x),\cdots,H_n(x)$. 

Given a cube $B_s$ in Definition \ref{OperT}, with $s\in\Gamma\setminus\{a\}$, and $1\leq i\leq n$ we define 
\begin{equation*}
\tt_{i,s}(x):=\dfrac{H_i(x-c_s)\chi_{s}(x)}{\int_{B_s}H_i(z-c_s):H_i(z-c_s)\dz},
\end{equation*}
where $c_s$ is the center of the cube $B_s$. Using the symmetries of the cubes $B_s$ which have sides parallel to the axis, it follows that $\int \tt_{s,i}(x):H_j(x)\dx=\delta_{i,j}$ for all $s\in\Gamma\setminus\{a\}$ and  $1\leq i, j\leq n$. Moreover, notice that $\int \tt_{s,i}(x):A_j(x)\dx=0$ for all $1\leq i\leq n$ and $1\leq j\leq n_0$. This property is basic to preserve the orthogonality obtained in the first step. 

We define the decomposition of $g$ in the following way:
\begin{align}\label{H-decomp}
\tg_t(x):=\tg^{(0)}_t(x)+\left(\sum_{s:\,s_p=t}\th^{(1)}_s(x)\right)-\th^{(1)}_t(x), 
\end{align}
where 
\begin{align}\label{h1}
\th^{(1)}_s(x):=\sum_{i=1}^{n}\left(\int_{W_s}H_i(y):\sum_{k\succeq s}\tg^{(0)}_k(y)\dy\right) \tt_{s,i}(x). 
\end{align} 
In the particular case when $t= a$,  (\ref{H-decomp}) means 

\begin{eqnarray*}
\tg_a(x)&:=&\tgc_a(x)+\sum_{s:\,s_p=a} \thu_s(x).
\end{eqnarray*}

As before, $\tg^{(0)}_s$ and $\th^{(1)}_s$ are identically zero if $s\in\Gamma\setminus\Gamma_{\tg}$ implying that $\tg_s$ is identically zero if $s\in\Gamma\setminus\Gamma_{\tg}$. 

In order to prove (\ref{P11}), notice that $supp(\thc_s)\subseteq B_s$ and $supp(\thu_s)\subseteq B_s$. Thus, from (\ref{H-decomp}) and (\ref{A-decomp}),  it follows that $\tg_t(x)=\tf_t(x)$ for all $x\not\in\bigcup_{s\in\Gamma} B_s$ or $x\in B_s$ with $s\neq t$ and $s_p\neq t$. Then, (\ref{P11}) is proved. 

Estimate (\ref{P12}) requires more work. Let us start by showing a pointwise estimate of $\|\thu_s\|_1(x)$ by the Hardy type operator $T$ on $\|\tg\|_1(x)$. First, notice that \[\|\tt_{s,i}\|_\infty(x)\leq \frac{6}{l_s^{n+1}}=\frac{6}{l_s|B_s|},\]
where $l_s$ is the side length of the cube $B_s$. Next, using the orthogonality of the collection $\{\tgc_k\}$ with respect to $A_1,\cdots,A_{n_0}$, we can conclude that \[\int_{W_s}H_i(y):\sum_{k\succeq s}\tg^{(0)}_k(y)\dy=\int_{W_s}H_i(y-c_s):\sum_{k\succeq s}\tg^{(0)}_k(y)\dy.\]
Thus, by replacing this new integral in definition (\ref{h1}) and using that $\|H_i\|_\infty(y-c_s)\leq \diam(W_s)$ for all $y\in W_s$, we have
\begin{align*}
\|\thu_s\|_\infty(x)&\leq\sum_{i=1}^{n}\|\tt_{s,i}\|_\infty(x)\int_{W_{s}}|H_i(y-c_s)|:|\sum_{k\succeq s}\tg^{(0)}_k(y)|\dy\\
&\leq \sum_{i=1}^{n} 6  \frac{\chi_s(x)}{l_s|B_s|}  \int_{W_{s}}\|H_i\|_\infty(y-c_s) \|\sum_{k\succeq s}\tg^{(0)}_k\|_1(y)\dy\\
&\leq \sum_{i=1}^{n} 6  \frac{\,\diam(W_s)}{l_s}\frac{\chi_s(x)}{|B_s|}  \int_{W_{s}}\|\sum_{k\succeq s}\tg^{(0)}_k\|_1(y)\dy\\
&= 6n \frac{\,\diam(W_s)}{l_s}  \frac{\chi_s(x)}{|B_s|}    \int_{W_{s}}\|\sum_{k\succeq s}\tg^{(0)}_k\|_1(y)\dy=(1)\\
\end{align*}
Now, it can be seen by using (\ref{A-decomp}) and certain telescopic cancelations of the functions $\thc_{k'}$ that 
\begin{align*}
\sum_{k\succeq s}\tgc_k=\sum_{k\succeq s}\left(\tilde{f}_k+\left(\sum_{k':\,k'_p=k}\th^{(0)}_{k'}\right)-\th^{(0)}_k\right)=\left(\sum_{k\succeq s}\tf_k\right)-\thc_s.
\end{align*}
Thus, from (\ref{esths}) it follows
\begin{align*}
(1)&= 6n \frac{\,\diam(W_s)}{l_s}  \frac{\chi_s(x)}{|B_s|}    \int_{W_{s}}\|\sum_{k\succeq s}\tf_k-\thc_s\|_1(y)\dy\\
&\leq 6n \frac{\,\diam(W_s)}{l_s}  \frac{\chi_s(x)}{|B_s|}    \int_{W_{s}}\|\sum_{k\succeq s}\tf_k\|_1(y)+\|\thc_s\|_1(y)\dy\\
&\leq 6n \frac{\,\diam(W_s)}{l_s}  \frac{\chi_s(x)}{|B_s|}    \int_{W_{s}}\|\tg\|_1(y)+\frac{n^2}{2|B_s|}\chi_s(y)\left(\int_{W_s}\|\tg\|_1(x)\dx\right)\dy\\
&= 6n\left(1+\frac{n^2}{2}\right) \frac{\,\diam(W_s)}{l_s}  \frac{|W_s|}{|B_s|}     \frac{\chi_s(x)}{|W_s|} \int_{W_{s}}\|\tg(y)\|_1\dy\\
&= C_n  \frac{\,\diam(W_s)}{l_s}  \frac{|W_s|}{|B_s|} T\| \tg\|_1(x).
\end{align*}
Hence, using (\ref{Boman tree}),
\begin{align}\label{P02+}
\|\thu_s\|_\infty(x)\leq C_n  K^{n+1} T\| \tg\|_1(x).
\end{align}
Finally, we have already mentioned that the functions $\thu_s$ and $\thu_t$ defined in  (\ref{H-decomp}) are supported, respectively, in the pairwise disjoint sets $B_s$ and $B_t$. Thus, combining (\ref{P02}) and (\ref{P02+}), we have that for any $x\in B_s$, where $s=t$ or $s_p=t$,
\begin{align*}
\|\tg_t\|_\infty(x)&\leq \|\tgc\|_\infty(x)+\|\thu_s\|_\infty(x)\\
&\leq \|\tg\|_\infty(x)+K^n\,T\|\tg\|_1(x)+C_n  K^{n+1} T\| \tg\|_1(x),
\end{align*}
proving (\ref{P12}). 

Properties $(1)$, $(2)$, and $(3)$ in the statement of this lemma follows by using the construction of the partition. Indeed, the first two properties follow by replacing $\tg_t$ by $\tg_t^{(0)}$  and $\tg^{(0)}_t$ by $\tg^{(1)}_t$ in the first step. 

The third property follows from the facts that $\int \th^{(1)}_s : A_i=0$ for all $s\in\Gamma$ and $1\leq i\leq n_0$, so we do not modify the orthogonality obtained in the previous step, and 
\[\int H_j(x) : \th^{(1)}_s(x) \,{\rm d}x=\int H_j(x) : \sum_{k\succeq s}\tg^{(0)}_k(x)\,{\rm d}x.\]
The rest of the proof follows by mimicking the first step.
\end{proof}

\begin{proof} [Proof of Lemma \ref{DecompSimple}]  As $\beta\geq 0$ and $\O$ is bounded, $L^q(\O,\Rnn,\rho^{-\beta})\subset L^1(\O,\Rnn)$. Thus, having just proved Lemma \ref{Decomp} we only need to show inequality (\ref{Estimation}). Hence, from (\ref{P12}) and (\ref{P11}), and the continuity of the operator $T$ stated in Lemma \ref{weighted T}, we have

\begin{align*}
\sum_{t\in\Gamma} \|\tg_t\|_{L^{q}(\O_t,\rho^{-\beta})}^q&=\sum_{t\in\Gamma} \int_{\O_t} \|\tg_t\|^q_q(x)\o^{-q}(x)\dx\\
&=C K^{q(n+1)}\sum_{t\in\Gamma} \int_{\O_t} \left(\|\tg\|^q_q(x)+(T\|\tg\|_1(x))^q\right)\o^{-q}(x)\dx\\
&=C K^{q(n+1)}\left(\|\tg\|^q_{L^q(\O,\rho^{-\beta})}+\|T\|\tg\|_1\|^q_{L^q(\O,\rho^{-\beta})}\right)\\
&=C K^{q(n+1+\beta)}\|\tg\|^q_{L^q(\O,\rho^{-\beta})}
\end{align*}
where $C$ is independent of $\O$.
\end{proof}

\section*{Acknowledgements}
The author thanks Marta Lewicka from University of Pittsburgh for bringing to his attention the version of Korn inequality studied in this work.

\end{document}